\documentclass[10pt]{amsart}

\usepackage{hyperref}
\usepackage{amsmath}
\usepackage{graphicx}
\usepackage{tikz}
\usepackage{tikz-cd}
\usepackage{tikz-3dplot}
\usepackage{subcaption}
\usepackage{enumitem}
\usetikzlibrary{snakes}
\usepackage[all]{xy}
\usepackage{framed}
\usepackage{mathrsfs}
\usepackage{amssymb}
\usepackage{stmaryrd}

\usepackage{mathabx}

\usepackage{colonequals} 

\theoremstyle{plain}
\newtheorem{theorem}{Theorem}[section]
\newtheorem{lemma}[theorem]{Lemma}

\theoremstyle{definition}
\newtheorem{definition}[theorem]{Definition}
\newtheorem{example}[theorem]{Example}
\newtheorem{remark}[theorem]{Remark}

\numberwithin{equation}{section}

\newcommand{\CC}{\mathbb{C}}
\newcommand{\CP}{\mathbb{C}{\text P}}
\newcommand{\QQ}{\mathbb{Q}}
\newcommand{\RR}{\mathbb{R}}
\newcommand{\ZZ}{\mathbb{Z}}

\newcommand{\mult}{{\rm mult}}

\DeclareMathOperator{\PD}{PD}

\begin{document}

\title{Cohomology bases of toric surfaces}

\author{Xin Fu} 
\address{Beijing Institute of Mathematical Sciences and Applications, Beijing 101408, China} 
\email{x.fu@bimsa.cn} 

\author{Tseleung So} 
\address{School of Mathematical Sciences, University of Southampton, Southampton SO17 1BJ, United Kingdom} 
\email{larry.so.tl@gmail.com} 

\author{Jongbaek Song}
\address{Department of Mathematics Education, Pusan National University,  Busan, 46241, Republic of Korea}
\email{jongbaek.song@pusan.ac.kr}

\thanks{
This work was supported by a New Faculty Research Grant of Pusan National University~2023. X.\,Fu was supported by the Beijing Natural Science Foundation (grant no.\,1244043).
}



\maketitle 
\abstract 
Given a compact toric surface, the multiplication of its rational cohomology 
can be described in terms of the intersection products of Weil divisors, or in terms of the cup products of cohomology classes representing specific cells.
In this paper, we aim to compare these two descriptions. More precisely, we define two different cohomology bases, the \emph{Poincar\'{e} dual basis} and the \emph{cellular basis}, which give rise to matrices representing the intersection product and the cup product. We prove that these representing matrices are inverse of each other.
\endabstract

\section{Introduction}
When working with the cohomology of a variety, choosing a good basis not only facilitates the computation but also reveals intrinsic geometric, topological, and algebraic properties of the given variety. This insight is especially crucial in the study of multiplication structures of cohomology rings, where different bases can provide complementary perspectives on the underlying geometry. This paper aims to investigate these phenomena in the context of certain toric varieties. 

Toric varieties, a particularly rich class of algebraic varieties, are known for their deep connections to topology, combinatorics, representation theory, algebraic geometry and symplectic geometry. 
Cohomology is one way to see these connections as its algebraic structure reflects the geometry of the variety, for example information about divisor classes and their intersections; see for instance \cite[Chapter~5]{Ful}.

Among all toric varieties, we focus on complete toric surfaces, which are complex two-dimensional normal toric varieties associated with complete fans $\Sigma$, or equivalently associated with $2$-dimensional rational polytopes $P$. Depending on the combinatorial description, we write the corresponding toric surface as $X_\Sigma$ or~$X_P$. These surfaces may contain singularities, which makes them difficult to work with. Nevertheless, they remain accessible since their singularities are at worst orbifold singularities and their cohomological properties over the rational coefficients resemble those of smooth varieties, such as the Poincar\'e duality.

Let $\Sigma$ be a complete 2-dimensional fan with $(n+2)$ rays $\rho_1,\ldots,\rho_{n+2}$. Applying a basis change if necessary, we always assume $\rho_{n+1}=(1,0)$.
Let $X_{\Sigma}$ be the toric surface associated with $\Sigma$.
Then $H^i(X_{\Sigma};\QQ)$ is $\QQ$ for $i=0,4$, is $\QQ^n$ for $i=2$, and is trivial otherwise. After fixing a generator of $H^4(X_{\Sigma};\QQ)$, we regard the cup product as an inner product
\[
H^2(X_{\Sigma};\QQ)\otimes H^2(X_{\Sigma};\QQ)\to H^4(X_{\Sigma};\QQ)\cong\QQ,
\]
and represent it by a symmetric matrix by choosing a basis for $H^2(X_{\Sigma};\QQ)$.

For example, given a compact toric surface $X_\Sigma$, 
there is a ring isomorphism
\[
(A^*(X_{\Sigma})_{\QQ}, \,\cdot \,)\overset{\cong}{\longrightarrow}(H^*(X_{\Sigma};\QQ),\cup)
\]
between the rational Chow ring and the rational cohomology ring for $X_{\Sigma}$, where~$\cdot$ is the intersection product.
It is known that the rational equivalence classes of Weil divisors corresponding to  rays~$\rho_1,\ldots,\rho_n$ and the subvariety corresponding to a 2-dimensional cone $\sigma$ in $\Sigma$, say  
\[
\{[D_{\rho_1}],\ldots,[D_{\rho_n}];[V(\sigma)]\},\]
forms a basis for $A^{>0}(X_{\Sigma})_{\QQ}$. The ring structure of $A^*(X_{\Sigma})_{\QQ}\cong H^*(X_{\Sigma};\QQ)$ is then determined by the matrix
\[
M_{int}(X_{\Sigma})=(m_{ij})_{1\leq i,j\leq n}
\]
where $m_{ij}$ is given by $[D_{\rho_i}]\cdot [D_{\rho_j}]=m_{ij}[V(\sigma)]$, which we call the \emph{intersection product matrix}. 
A brief review on toric surfaces and their Chow rings is given in Section~\ref{sec_toric_surf}.

On the other hand, in recent work~\cite{FSS2} the authors combined ideas from homotopy theory and toric topology to derive a new formula for computing cup products in $H^*(X_{\Sigma};\QQ)$.\footnote{In fact, this formula holds for 4-dimensional toric orbifolds and a wider class of coefficients.} Roughly speaking, we defined a specific cellular structure for $X_{\Sigma}$, and used these cells to construct elements $u_i\in  H^2(X_{\Sigma};\QQ)$ for $i=1, \dots, n$ and $v\in H^4(X_{\Sigma};\QQ)$ such that 
\[
\{u_1, \dots, u_n ; v\}
\]
forms a basis for $\widetilde{H}^\ast(X_\Sigma; \QQ)$, which is called the \emph{cellular basis}.
The cup product of~$H^*(X_{\Sigma};\QQ)$ is then represented by the \emph{cellular cup product matrix}
\[
M_{cup}(X_{\Sigma})=(c_{ij})_{1\leq i,j\leq n}
\]
where $c_{ij}$ satisfies $u_i\cup u_j=c_{ij}v$, which was worked out in \cite[Theorem~1.3]{FSS2}. The construction of cellular bases and cellular cup product matrices will be described in Section~\ref{sec_cellular cohomology} and Appendix~\ref{Appendix_alg_cel_basis}.

This article aims to compare these two approaches by studying the relation between these two bases and their corresponding matrices. Our main result is Theorem~\ref{thm_main} and will be proved in Section~\ref{sec_cellular and poincare}.

\begin{theorem}[restated in Theorem~\ref{thm_main}]\label{thm_main_Intro}
Let $X_{\Sigma}$ be a toric surface associated with a 2-dimensional complete fan $\Sigma$.
Let $M_{int}(X_\Sigma)$ and $M_{cup}(X_\Sigma)$ be the intersection matrix and the cellular cup product matrix respectively. Then we have
\begin{equation}\label{eqn_main thm}
M_{int}(X_\Sigma)\cdot M_{cup}(X_\Sigma)=I.
\end{equation}
\end{theorem}

Due to Theorem~\ref{thm_main_Intro}, two bases $\{u_1,\ldots,u_n;v\}$ and $\{[D_{\rho_1}],\ldots,[D_{\rho_n}];[V(\sigma)]\}$ can be interpreted as a pair of dual bases in the following sense.
Using the equivalence class $[V(\sigma)]$ to identify $A^2(X_{\Sigma})_{\QQ}\cong\QQ$, we regard the pair $(A^*(X_{\Sigma})_{\QQ}, \,\cdot \, )$ as an inner product space and $M_{int}(X_{\Sigma})$ as the matrix representation of intersection product with respect to the basis $\{[D_{\rho_1}],\ldots,[D_{\rho_n}]\}$. Similarly, $(H^*(X_{\Sigma};\QQ),\cup)$ is an inner product space and $M_{cup}(X_{\Sigma})$ is the matrix representation of $\cup$ with respect to the cellular basis.
From this perspective, Equation~\eqref{eqn_main thm} implies that $(A^*(X_P)_{\QQ},\, \cdot\,)$ is the dual inner product space of $(H^*(X_P;\QQ),\cup)$.

\subsection*{Acknowledgment}
The authors would like to express our gratitude to Mikiya Masuda for bringing them the insightful question that ultimately motivated this work. The second author would like to thank Matthias Franz and Martin Pinsonnault for communications regarding his questions.

\section{Toric varieties and Chow rings}\label{sec_toric_surf}
\subsection{Toric varieties and orbit closures}
For an $n$-dimensional compact torus $T^n=(S^1)^n$, let $N$ be the lattice of one-parameter subgroups $S^1\to T^n$, and $M$ the lattice of characters $T^n\to S^1$.  For simplicity, we identify both~$N$ and $M$ with the standard lattice $\ZZ^n$ throughout this paper. Under this identification, the one parameter subgroup corresponding to $\mathbf{a}=(a_1, \dots, a_n)\in \ZZ^n$ is the map 
\[
\lambda^\mathbf{a} \colon S^1\to T^n,\quad t\mapsto (t^{a_1}, \dots, t^{a_n}).\] 
The character corresponding to $\mathbf{b}=(b_1, \dots, b_n)\in \ZZ^n$ is the map 
\[\chi^\mathbf{b} \colon T^n\to S^1,\quad
(t_1, \dots, t_n)\mapsto t_1^{b_1}\cdots t_n^{b_n}.\] Then the natural pairing $\left< ~,~\right> \colon M\times N \to \ZZ$ can be understood as the standard inner product of $\ZZ^n$. 

Given a full-dimensional rational polytope $P$ in $M\otimes_\ZZ \RR$ and its normal fan $\Sigma$ in~$N\otimes_\ZZ \RR$, we denote by $X_\Sigma$ (or $X_P$) the the toric variety associated with $\Sigma$ (or~$P$)\footnote{In general, a fan $\Sigma$ need not be the normal fan of a polytope $P$. However, we restrict our attention to this case for our main result.}. 
For each face $E$ of $P$, we denote by $\sigma_E$ the cone in $\Sigma$ generated by the primitive outward normal vectors of the facets of $P$ containing $E$ as a face. 
By translating~$P$ appropriately so that the origin in $M$ is a vertex of $E$, we may regard $E$ as a full-dimensional lattice polytope in $(\sigma_E^\perp \cap M)\otimes_\ZZ \RR$, where 
\[
\sigma_E^\perp\colonequals \{\mathbf{b}\in M\otimes_\ZZ \RR\mid \left<\mathbf{b}, \mathbf{a}\right>=0 \text{ for all } \mathbf{a}\in \sigma_E\}. 
\]
Let $N(\sigma_E)$ and $N_{\sigma_E}$ denote the lattice dual to $\sigma_E^\perp\cap M$ and the lattice generated by the vectors in~$\sigma_E\cap N$, respectively. We define the normal fan $\Sigma_E$ of $E$ in 
\[
N(\sigma_E)\otimes_\ZZ \RR \cong (N\otimes_\ZZ \RR)/ (N_{\sigma_E}\otimes_\ZZ \RR).
\] 
The associated toric variety of~$\Sigma_E$ is 
denoted as $X_E$.

Due to the one-to-one correspondence between the faces of $P$ and the cones in~$\Sigma$, we shall write $V(\sigma)\colonequals X_E$ where $\sigma=\sigma_E$ is defined as above. In particular, if $E$ is a facet of $P$, then $X_E$ is a divisor of $X_\Sigma$. In this case, we shall write $D(\rho)\colonequals X_E$ where $\rho= \sigma_E$ is the $1$-dimensional cone corresponding to $E$.
We call $V(\sigma)$ or $D(\rho)$ the \emph{orbit closures} corresponding to $\sigma$ or $\rho$, respectively. 

\begin{remark}
The usual definition of the orbit closure corresponding to a cone in~$\Sigma$ comes with the orbit decomposition of the affine toric variety associated with the given cone with respect to the algebraic torus action. The reader is referred to \cite[Chapter 3.2]{CLS} for more details. 
\end{remark}

The toric variety $X_\Sigma$ associated with a simple rational polytope has at worst orbifold singularities; see for instance \cite[Theorem 11.4.8]{CLS}. In the literature, such a variety is often called a \emph{complete simplicial toric variety} or a \emph{toric orbifold}. 
In particular, when $P$ is a rational polygon in $M\otimes_\ZZ \RR$ for a $2$-dimensional lattice~$M$ and $\Sigma$ is the corresponding complete $2$-dimensional fan in $N\otimes_\ZZ \RR$,
the associated toric variety $X_\Sigma$ is called a \emph{toric surface}. 

\subsection{Chow rings and intersection product matrices}\label{sec intersection product}
Let $X_\Sigma$ be an arbitrary toric variety associated with a fan $\Sigma$. The Chow group~$A^*(X_\Sigma)$ of~$X_\Sigma$ is generated by the rational equivalence classes $[V(\sigma)]$ of orbit closures $V(\sigma)$ corresponding to cones $\sigma$ in  $\Sigma$. 
To be more precise, if $\sigma$ is of dimension $k$, then $V(\sigma)$ is an $(n-k)$-dimensional subvariety of $X_\Sigma$, and thus $[V(\sigma)]$ is a class in $A^{k}(X_\Sigma)$. 

For the case of complete simplicial toric varieties, the rational Chow group 
\[
A^*(X_\Sigma)_{\mathbb{Q}}\colonequals  \bigoplus_{k=0}^n A^k(X_\Sigma)\otimes \mathbb{Q}
\]
has the intersection product
\[
A^k(X_\Sigma)_{\mathbb{Q}}\otimes A^\ell(X_\Sigma)_\QQ\to A^{k+\ell}(X_\Sigma)_{\mathbb{Q}},
\]
given by 
\begin{equation}\label{eq_intersection_prod}
[V(\sigma)]\cdot [V(\tau)] = 
\frac{\mult(\sigma)\cdot \mult(\tau)}{\mult(\gamma)} [V(\gamma)]
\end{equation}
for cones $\sigma,\tau$ and $\gamma$ such that $\sigma+ \tau= \gamma$, where the sum $\sigma+\tau$ denotes the cone generated by the $1$-dimensional cones in $\sigma$ and $\tau$. The \emph{multiplicity} $\mult(\sigma)$ of $\sigma$ is the index of the sublattice generated by the primitive vectors spanning the $1$-dimensional rays in~$\sigma$. The product is defined to be $0$ if $V(\sigma)$ and $V(\tau)$ do not intersect. 

The intersection product in \eqref{eq_intersection_prod} implies that as a graded $\mathbb{Q}$-algebra~$A^*(X_\Sigma)_{\mathbb{Q}}$ is generated by the classes $[D(\rho)]\in A^1(X_\Sigma)_\mathbb{Q}$, where $D(\rho)$'s are $\mathbb{Q}$-Cartier divisors 
for $\rho\in \Sigma^{(1)}$, the set of rays of $\Sigma$. 
Each $[D(\rho)]$ defines a rational homology class in~$H_{2n-2}(X_\Sigma;\QQ)$, whose Poincar\'e dual $x_\rho$ is contained in $H^2(X_\Sigma;\mathbb{Q})$. 
Additionally, the assignment $[D(\rho)]\mapsto x_\rho$ yields a ring isomorphism 
\begin{equation}\label{eq_chow_to_singular}
\psi\colon  A^\ast(X_\Sigma)_\QQ \to H^\ast(X_\Sigma;\QQ).
\end{equation}
 Moreover, it is known from \cite{Dan, Jur} that the rational cohomology ring of~$X_\Sigma$ is isomorphic to the following quotient polynomial ring
\begin{equation}\label{eq_rational_cohom_SR_presentation}
\QQ[x_\rho \mid \rho\in \Sigma^{(1)}]/\left(\mathcal{I}+\mathcal{J}\right)
\end{equation}
where 
$\mathcal{I}$ is the ideal generated by square-free monomials 
\[\left\{ \prod_{\alpha\in J}x_{\rho_\alpha} \Bigm|  
\sum_{\alpha\in J}\rho_\alpha \notin \Sigma \right\}\]
and $\mathcal{J}$ is the ideal generated by linear combinations 
\[\left\{ \sum_{\rho \in \Sigma^{(1)}} \left< m, \lambda_\rho \right> \Bigm| m \in M  \right\}\]
where $\lambda_\rho\in N$ denotes the primitive vector generating $\rho \in \Sigma^{(1)}.$

When $X_\Sigma$ is a toric surface, namely $\Sigma$ is a normal fan of a lattice polygon $P$, we assume that $P$ has $n+2$ facets say $F_1, \dots, F_{n+2}$ labeled in a counterclockwise direction. Identifying the lattice $N$ with $\mathbb{Z}^2$ by a suitable choice of a basis, we write  
\[
\lambda_i=(a_i, b_i)
\]
as the primitive outward normal vector of $F_i$, which spans the $1$-dimensional cone $\rho_i$ in $\Sigma$. Then, 
due to \eqref{eq_chow_to_singular} and \eqref{eq_rational_cohom_SR_presentation}  there is an isomorphism 
\begin{equation}\label{Eq_chow_SR_surface}
\varphi \colon A^\ast(X_\Sigma)_\mathbb{Q}\to \mathbb{Q}[x_1, \dots, x_{n+2}]/( \mathcal{I}+\mathcal{J})    
\end{equation}
defined by $\varphi([D(\rho_i)]) = x_i$,
and the ideals $\mathcal{I}$ and $\mathcal{J}$ are given by
\begin{align}
\begin{split}\label{eq_I_J_for_toric_surf}
\mathcal{I}&=\left\langle a_1x_1+\cdots + a_{n+2}x_{n+2},~ b_1x_1+\cdots + b_{n+2}x_{n+2}\right\rangle; \\
\mathcal{J}&=
\left\langle x_ix_j \mid  
i\neq j \text{ and } |i-j| \not\equiv 1 \text{ mod } (n+2)\right\rangle.
\end{split}
\end{align}
Then the intersection product \eqref{eq_intersection_prod} of the Chow ring $A^\ast(X_\Sigma)_\mathbb{Q}$ simplifies to 
\begin{equation}\label{eq_intersection_prod_2d}
[D(\rho_i)]\cdot [D(\rho_j)]=\begin{cases} 
\frac{1}{\mult(\sigma)}[V(\sigma)]& \text{if } |i-j| \equiv 1 ; \\[0.5em]
 \frac{c_i}{c_i'\cdot \mult(\sigma')}[V(\sigma')]
 =\frac{c_i}{c_i''\cdot \mult(\sigma'')}[V(\sigma'')] &\text{if } i=j;\\[0.5em]
0  &\text{otherwise.} 
\end{cases}
\end{equation}
Here, in the first case $\sigma=\rho_i+\rho_j$ and $\equiv$ means the congruence of modulo $n+2$, and in the second case $\sigma'=\rho_{i-1}+ \rho_i$ and
$\sigma''=\rho_i+\rho_{i+1}$. 

Writing $\lambda_i=(a_i, b_i)\in N$ as the primitive vector spanning $\rho_i$, it follows that 
\begin{equation}\label{eq multiplicity}
\mult(\sigma')=a_{i-1}b_{i}-a_{i}b_{i-1}
\text{ and }
\mult(\sigma'')=a_ib_{i+1}-a_{i+1}b_i. 
\end{equation}
Moreover, consecutive three vectors $\lambda_{i-1},\lambda_i,\lambda_{i+1}$ satisfy that 
\begin{equation}\label{eq_wall_rel}
c_i'\lambda_{i-1}+c_i\lambda_i+c_i''\lambda_{i+1}=0
\end{equation}
for some integers $c_i', c_i, c_i''\in \ZZ$. 
Since $\rho_i$ is located between $\rho_{i-1}$ and $\rho_{i+1}$ in the counterclockwise direction, we may assume that $c_i'\cdot c_i'' >0$. Equation~\eqref{eq_wall_rel} is called a \emph{wall relation}; see~\cite[(6.4.4)]{CLS}.

\begin{lemma}\label{lem_Chow_deg_4_generator}
Let $\Sigma$ be a complete $2$-dimensional fan and  $X_\Sigma$ the associated toric surface. 
For any two distinct 2-dimensional cones $\sigma, \tau$ in $\Sigma$, we have
\[\text{$[V(\sigma)]=[V(\tau)]$ in $A^2(X_\Sigma)_\mathbb{Q}$.}\] 
\end{lemma}
\begin{proof}
For $1\leq i\leq n+2$,
let $\sigma_i=\rho_i+\rho_{i+1}$ denote the $2$-dimensional cone spanned 
by~$\rho_i, \rho_{i+1}\in \Sigma^{(1)}$ with the convention $\rho_{n+3}=\rho_1$. 
It suffices to show that $[V(\sigma_{i-1})]=[V(\sigma_{i})]$ for all $i=2, \dots, n$.

Write $A\colonequals  a_1x_1+\cdots + a_{n+2}x_{n+2}$ and $B\colonequals b_1x_1+\cdots + b_{n+2}x_{n+2}$ as the two generators of $\mathcal{I}$ in \eqref{eq_I_J_for_toric_surf}.
Then we have
\[
[a_ix_iB-b_ix_iA]  
=[(a_{i}b_{i-1}-a_{i-1}b_i) x_{i-1}x_i -(a_{i+1}b_{i}-a_{i}b_{i+1})x_{i}x_{i+1}]
\]
in $\mathbb{Q}[x_1, \dots, x_{n+2}]/( \mathcal{I}+\mathcal{J})$. 
Since $a_ix_iB-b_ix_iA\in \mathcal I$, we have
\begin{equation}\label{eq identity}
\left[(a_{i}b_{i-1}-a_{i-1}b_i) x_{i-1}x_i\right] = \left[ 
  (a_{i+1}b_{i} -a_{i}b_{i+1})x_{i}x_{i+1} \right].
\end{equation}
As the isomorphism $\varphi$ given in  \eqref{Eq_chow_SR_surface} is multiplicative, we have
\[\begin{split}
\mult(\sigma_{i-1})x_{i-1}x_{i}&=\mult(\sigma_{i-1})\varphi\left([D(\rho_{i-1})]\cdot [D(\rho_{i})]\right) = \varphi\left([V(\sigma_{i-1})]\right);\\
\mult(\sigma_i)x_{i}x_{i+1}&=\mult(\sigma_i)\varphi\left([D(\rho_i)]\cdot [D(\rho_{i+1})]\right) = \varphi\left([V(\sigma_i)]\right).
\end{split}
\]
Equations~\ref{eq multiplicity} and \eqref{eq identity} together imply that $\varphi\left([V(\sigma_{i-1})]\right)=\varphi\left([V(\sigma_i)]\right)$. Therefore we have 
 $[V(\sigma_{i-1})]= [V(\sigma_{i})]$.
\end{proof}

By Lemma \ref{lem_Chow_deg_4_generator}, any $[V(\sigma)]$ for a $2$-dimensional cone $\sigma\in \Sigma$ can be chosen as a generator of $A^2(X_\Sigma)_\mathbb{Q}\cong \mathbb{Q}$, which we denote by $[V]$.  
Hence, the set 
\begin{equation}\label{eq_PD_def}
\left\{[D(\rho_1)], \dots, [D(\rho_n)]; [V]\right\}
\end{equation}
forms a basis of $A^{>0}(X_\Sigma)_\mathbb{Q}$. 
The intersection product 
\[
A^1(X_\Sigma)_\QQ \otimes A^1(X_\Sigma)_\QQ \to A^2(X_\Sigma)_\QQ\cong \QQ,
\]
with respect to the basis \eqref{eq_PD_def} is represented by the symmetric matrix 
\begin{equation}\label{eq_M_int_mij}
M_{int}(X_\Sigma)=(m_{ij})_{1\leq i,j\leq n}
\end{equation}  
where $m_{ij}=m_{ji}\in \mathbb{Q}$ satisfies $[D(\rho_i)]\cdot [D(\rho_j)]=m_{ij}[V]$. The explicit formula for $m_{ij}$ is given in \eqref{eq_intersection_prod_2d}. We call the matrix $M_{int}(X_\Sigma)$ of \eqref{eq_M_int_mij} the \emph{intersection product matrix} of $X_\Sigma$.

\begin{example}\label{ex_surface}
Let $\Sigma$ be the fan described in Figure \ref{fig_2dim_fan_ex}. The primitive vectors generating the $1$-dimensional cones of $\Sigma$ are 
\[
\lambda_1=(-2,1),\, \lambda_2=(-2,-1),\, \lambda_3=(1,-2),\, \lambda_4=(1,0) \text{ and }
\lambda_5=(0,1).\]
Additionally, we use the following wall relations:
\begin{align*}
-2\lambda_5 +\lambda_1 -\lambda_2 &=-2 (0,1) +(-2, 1)-(-2,-1)=(0,0),\\
-5\lambda_1 +3\lambda_2 -4\lambda_3 &=5(2,-1) -3(2,1)- 4(1,-2)=(0,0),\\
-2\lambda_2 + \lambda_3 -5 \lambda_4 &=2(2,1)+(1,-2) -5(1,0)=(0,0).
\end{align*}
Using the above relations and \eqref{eq_intersection_prod_2d}, we compute  the intersection products:
\begin{itemize}
\item $[D(\rho_1)]\cdot [D(\rho_1)]=-\frac{1}{4}[V]$,
\item $[D(\rho_2)]\cdot [D(\rho_2)]=-\frac{3}{20}[V]$,
\item $[D(\rho_3)]\cdot [D(\rho_3)]=-\frac{1}{10}[V]$,
\item $[D(\rho_1)]\cdot [D(\rho_2)]=[D(\rho_2)]\cdot [D(\rho_1)]=\frac{1}{4}[V]$,
\item $[D(\rho_1)]\cdot [D(\rho_3)]=[D(\rho_3)]\cdot [D(\rho_1)]=0$,
\item $[D(\rho_2)]\cdot [D(\rho_3)]=[D(\rho_3)]\cdot [D(\rho_2)]=\frac{1}{5}[V]$.
\end{itemize}
Hence the intersection product matrix $M_{int}(X_\Sigma)$ is given by 
\[
M_{int}(X_\Sigma)=
\begin{pmatrix}
-\frac{1}{4}&\frac{1}{4}&0\\
\frac{1}{4}&-\frac{3}{20}&\frac{1}{5}\\
0&\frac{1}{5}&-\frac{1}{10}
\end{pmatrix}.
\]
\end{example}
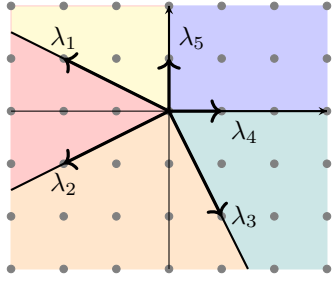
\begin{figure}
\begin{tikzpicture}[scale=0.7]

\draw[blue!20, fill=blue!20] (0,0)--(3,0)--(3,2)--(0,2)--cycle;
\draw[red!20, fill=yellow!20] (0,0)--(0,2)--(-3,2)--(-3,1.5)--cycle;
\draw[red!20, fill=red!20] (0,0)--(-3,1.5)--(-3,-1.5)--cycle;
\draw[orange!20, fill=orange!20] (0,0)--(-3,-1.5)--(-3,-3)--(1.5,-3)--cycle;
\draw[teal!20, fill=teal!20] (0,0)--(1.5,-3)--(3,-3)--(3,0)--cycle;

\foreach \x in {-3, -2,-1,0,1,2,3}{
	      \foreach \y in {-3, -2,-1,0,1,2}{
	        \node[draw,circle,inner sep=1pt,fill=gray, gray] at (\x,\y) {};
      }
    }

\draw[-stealth] (-3,0)--(3,0);
\draw[-stealth] (0,-3)--(0,2); 
\draw[very thick, ->] (0,0)--(1,0); 
\draw[thick] (0,0)--(3,0);
\node[below right] at (1,0) {$\lambda_4$};

\draw[very thick, ->] (0,0)--(0,1); 
\draw[thick] (0,0)--(0,2); 
\node[above right] at (0,1) {$\lambda_5$};

\draw[very thick, ->] (0,0)--(-2,1); 
\draw[thick] (0,0)--(-3,1.5); 
\node[above] at (-2,1) {$\lambda_1$};

\draw[very thick, ->] (0,0)--(-2,-1); 
\draw[thick] (0,0)--(-3,-1.5); 
\node[below] at (-2,-1) {$\lambda_2$};

\draw[very thick, ->] (0,0)--(1,-2); 
\draw[thick] (0,0)--(1.5,-3); 
\node[right] at (1,-2) {$\lambda_3$};

\end{tikzpicture}
\caption{An example of a complete $2$-dimensional fan $\Sigma$.}
\label{fig_2dim_fan_ex}
\end{figure}

\section{Topological construction and Cellular cohomology bases}\label{sec_cellular cohomology}
\subsection{Topological construction of a toric variety}
Given a lattice polytope $P\subset M\otimes_\ZZ \RR$, let $\mathcal{F}(P)$ be the set of facets of $P$. For each $F\in\mathcal{F}(P)$, we denote by~$\lambda_F\in N$ the primitive outward normal vector of the hyperplane supporting~$F$. One can associate a face $E$ of $P$ with the subtorus $T_E$ generated by the vectors $\lambda_F$ for  $F\in \mathcal{F}(P)$ such that $E\subset F$.

In this section we write $X_P$ as the toric variety associated with $P$. Following Jukiewicz's work \cite{Jur} on the moment map $X_P \to \mathfrak{t}^\ast$, where $\mathfrak{t}^\ast$ is the dual Lie algebra of $T^n$, identified with $M\otimes_\ZZ \RR$, there is a homeomorphism
\begin{equation}\label{eq_top_model}
X_P \cong P\times T^n /_\sim
\end{equation}
where $(p,t)\sim (q,s)$ if and only if $p=q$ and $t^{-1}s\in T_E$ for the face $E$ containing $p=q$ in its relative interior. We also refer to \cite{Fr}. In this case, the action of the torus $T^n$ can be understood as the natural multiplication of the second factor. Hence, the orbit map 
\[
\pi \colon  P\times T^n /_\sim \to P
\]
is the projection onto the first factor. Then, the orbit closure $V(\sigma)$ for the cone $\sigma$ in the normal fan $\Sigma$ of $P$ corresponding to $E$ is homeomorphic to $\pi^{-1}(E)$.

Now we provide a more detailed description for the case of toric surfaces. Let~$P$ be a rational polygon with $n+2$ facets $F_1, \dots, F_{n+2}$, labeled in a counterclockwise direction.
Recall from Section \ref{sec intersection product} that we write $\lambda_i = (a_i, b_i)$ for the primitive outward normal vector of $F_i$ for $i=1, \dots, n+2$, which spans the the corresponding $1$-dimensional cone $\rho_i\in \Sigma^{(1)}$. 
Note that each $2$-dimensional cone is generated by~$\{\lambda_i,\lambda_{i+1}\}$ for some $i\in\{1,\ldots,n+2\}$, which we denote by $\sigma_i$. Here $\lambda_{n+3}$ means~$\lambda_1$.

From the topological model \eqref{eq_top_model} of a toric variety $X_P$, one can see that $D(\rho_i)$ is homeomorphic to $\pi^{-1}(F_i)$ which is a $2$-sphere. We will denote this by $S^2_i$ following the notation from \cite{FSS}.
The orbit closure $V(\sigma_i)$ corresponding to a $2$-dimensional cone $\sigma_i$ is the fixed point of $X_P$ corresponding to the vertex $F_i \cap F_{i+1}$. See Figure~\ref{fig_P_Sigma_orbit_closure} for the description of a rational polygon, its normal fan and the corresponding orbit closures.

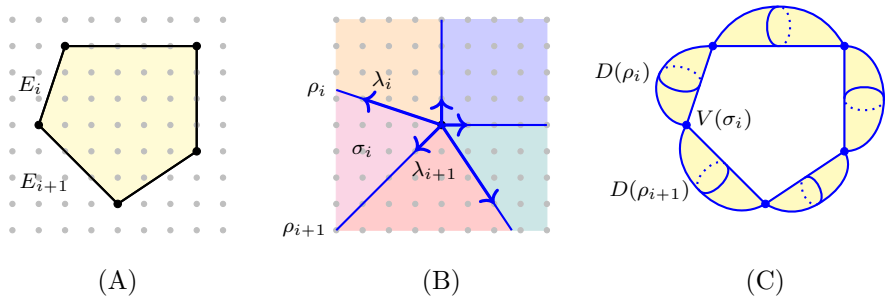
\begin{figure}
\begin{tikzpicture}[scale=0.35, rotate=180]

\begin{scope}[xshift=350]

\draw[thick, fill=yellow!20] (-3,-3)--(2,-3)--(3,0)--(0,3)--(-3,1)--cycle;

\foreach \x in {-4,-3,-2,-1,0,1,2,3,4}{
\foreach \y in {-4,-3,-2,-1,0,1,2,3,4}{
\node[draw,circle,inner sep=0.7pt,fill=gray!50, gray!50] at (\x,\y) {};
}
}
\draw[thick] (-3,-3)--(2,-3)--(3,0)--(0,3)--(-3,1)--cycle;

\node[left] at (2.5, -1.5) {\footnotesize$E_i$};
\node[below left] at (1.5, 1.5) {\footnotesize$E_{i+1}$};

\node[draw,circle,inner sep=1pt, fill=black] at (-3,-3) {};
\node[draw,circle,inner sep=1pt, fill=black] at (2,-3) {};
\node[draw,circle,inner sep=1pt, fill=black] at (3,0) {};
\node[draw,circle,inner sep=1pt, fill=black] at (0,3) {};
\node[draw,circle,inner sep=1pt, fill=black] at (-3,1) {};
\node at (0,6) {(A)};
\end{scope}

\begin{scope}

\draw[red!20, fill=red!20] (0,0)--(4,4)--(-8/3,4)--cycle;
\draw[teal!20, fill=teal!20] (0,0)--(-8/3,4)--(-4,4)--(-4,0)--cycle;
\draw[blue!20, fill=blue!20] (0,0)--(-4,0)--(-4,-4)--(0,-4)--cycle;
\draw[orange!20, fill=orange!20] (0,0)--(0,-4)--(4,-4)--(4,-4/3)--cycle;
\draw[magenta!20, fill=magenta!20] (0,0)--(4,-4/3)--(4,4)--cycle;
\node at (0,6) {(B)};

\foreach \x in {-4,-3,-2,-1,0,1,2,3,4}{
\foreach \y in {-4,-3,-2,-1,0,1,2,3,4}{
\node[draw,circle,inner sep=0.7pt,fill=gray!50, gray!50] at (\x,\y) {};
}
}


\node[draw,circle,inner sep=1pt, fill=black] at (0,0) {};

\draw[very thick, blue, ->] (0,0)--(1,1);
\node[below, right] at (1.5,1.75) {\footnotesize$\lambda_{i+1}$};
\draw[thick, blue] (0,0)--(4,4);
\node[left] at (4,4) {\footnotesize$\rho_{i+1}$};

\draw[very thick, blue, ->] (0,0)--(-2,3);
\draw[thick, blue] (0,0)--(-8/3, 4);

\draw[very thick, blue, ->] (0,0)--(-1,0);
\draw[thick, blue] (0,0)--(-4, 0);

\draw[very thick, blue, ->] (0,0)--(0,-1);
\draw[thick, blue] (0,0)--(0,-4);

\draw[very thick, blue, ->] (0,0)--(3,-1);
\node[above right] at (3,-1) {\footnotesize$\lambda_i$};
\draw[thick, blue] (0,0)--(4, -4/3);
\node[left] at (4,-4/3) {\footnotesize$\rho_i$};

\node at (3,1) {\footnotesize$\sigma_i$};
\end{scope}

\begin{scope}[xshift=-350]
\draw (-3,-3)--(2,-3)--(3,0)--(0,3)--(-3,1)--cycle;

\draw[thick, blue, fill=yellow!30] (-3,-3)--(2,-3) [out=240, in=0] to (-1/2,-4.5) [out=180,in=300] to (-3,-3);
\draw[thick, blue] (-1/2, -3) [out=0,in=0] to (-1/2, -4.5);
\draw[thick, blue, dotted] (-1/2, -3) [out=180,in=180] to (-1/2, -4.5);

\draw[thick, blue, fill=yellow!30] (2,-3)--(3,0) [out=0, in=60] to (4,-2) [out=240,in=0] to (2,-3);
\draw[thick, blue] (5/2, -3/2) [out=60,in=60] to (4,-2);
\draw[thick, blue, dotted] (4,-2) [out=240,in=240] to (5/2, -3/2);
\node[left] at (4,-2) {\footnotesize$D(\rho_i)$};

\draw[thick, blue, fill=yellow!30] (3,0)--(0,3) [out=30, in=135] to (2.5,2.5) [out=315, in=60] to (3,0);
\draw[thick, blue] (1.5,1.5) [out=135, in=135] to (2.5,2.5); 
\draw[thick, blue, dotted] (2.5,2.5) [out=315, in=315] to (1.5,1.5); 
\node[left] at (2.5,2.5) {\footnotesize$D(\rho_{i+1})$};

\draw[thick, blue, fill=yellow!30] (0,3)--(-3,1) [out=120, in=210] to (-2,3) [out=30, in=150] to (0,3);
\draw[thick, blue] (-2,3) [out=210, in=210] to (-1.5,2); 
\draw[thick, blue, dotted] (-2,3) [out=30, in=30] to (-1.5,2); 

\draw[thick, blue, fill=yellow!30] (-3,-3)--(-3,1) [out=180, in=90] to (-4.5,-1) [out=270,in=180] to (-3,-3);
\draw[thick, blue] (-4.5,-1) [out=270, in=270] to (-3,-1); 
\draw[thick, blue, dotted] (-4.5,-1) [out=90, in=90] to (-3,-1); 

\node[draw,circle,inner sep=1pt, fill=blue, blue] at (-3,-3) {};
\node[draw,circle,inner sep=1pt, fill=blue, blue] at (2,-3) {};
\node[draw,circle,inner sep=1pt, fill=blue, blue] at (3,0) {};
\node[right] at (3,-0.2) {\footnotesize$V(\sigma_i)$};
\node[draw,circle,inner sep=1pt, fill=blue, blue] at (0,3) {};
\node[draw,circle,inner sep=1pt, fill=blue, blue] at (-3,1) {};

\node at (0,6) {(C)};
\end{scope}
\end{tikzpicture}
\caption{(A) A rational polygon $P$ in $M\otimes_\ZZ \RR$; (B) The normal fan $\Sigma$ of $P$ in $N\otimes_\ZZ\RR$; (C) orbit closures.}
\label{fig_P_Sigma_orbit_closure}
\end{figure}

Note that $X_P$ can be written as the union
\[
X_P\cong\pi^{-1}(F_1\cup\cdots\cup F_n)\cup\pi^{-1}(P\setminus(F_1\cup\cdots\cup F_n)).
\]
On the right hand side, the first preimage is homotopy equivalent to the wedge of 2-spheres $S^2_1 \vee \cdots \vee S^2_n$. 
The second one is homeomorphic to $(\RR^2_\geq \times T^2)/_\sim$ as~$P\setminus (F_1\cup \cdots\cup F_n)$ is homeomorphic to the nonnegative quadrant~$\RR^2_\geq$ as manifolds with corners. For $\lambda_{n+1}=(a_{n+1}, b_{n+1})$ and $\lambda_{n+2}=(a_{n+2}, b_{n+2})$ corresponding to facets $F_{n+1}$ and~$F_{n+2}$, 
let $\phi \colon T^2\twoheadrightarrow T^2$ be the endomorphism defined by 
\[
\phi(t_1, t_2)=(t_1^{a_{n+1}}t_2^{a_{n+2}}, t_1^{b_{n+1}}t_2^{b_{n+2}}).
\]
Then the quotient space $(\RR^2_\geq \times T^2)/_\sim$ is homeomorphic to $\CC^2/\ker \phi$, where $\ker\phi$ acts on $\CC^2$ as a subgroup of $T^2$. Therefore, one can also view $X_P$ as the homotopy cofiber in the homotopy cofibration sequence
\begin{equation}\label{eq_cofib}
S^3/\ker \phi \to \bigvee_{i=1}^n S^2_i \to X_P;
\end{equation}
see for instance \cite{BSS, FSS} for more details.

\subsection{Cellular bases and cellular cup product matrices}
The homotopy cofibration sequence  \eqref{eq_cofib} implies that, rationally,
$X_P$ has a cellular structure consisting of a collection of 2-spheres $S^2_i$ and a single 4-cell.
Therefore, the rational homology group of $X_P$ has a basis 
\begin{equation}\label{eq_cellular_basis}
\{e_i \mid i=1, \dots, n\}\subset H_2(X_P;\QQ) \quad \text{and} \quad f \in H_4(X_P;\QQ),
\end{equation}
where $e_i$ is the homology class corresponding to the 2-sphere $S^2_i$ and $f$ corresponds to the 4-cell of $X_P$. 

\begin{definition}\cite[Definition 3.2]{FSS}\label{def_M_cup}
Let $u_i\in H^2(X_P;\QQ)$ be the Kronecker dual of $e_i\in H_2(X_P;\QQ)$ and $v\in  H^4(X_P;\QQ)$ the Kronecker dual of $f\in H_4(X_P;\QQ)$, respectively. 
\begin{enumerate}
    \item The set $\{u_1,\ldots,u_n;v\}$ is called the \emph{cellular basis} of~$\widetilde{H}^*(X_P;\QQ)$.
    \item The symmetric matrix~$M_{cup}(X_P)=\left(c_{ij}\right)_{1\leq i,j \leq n}$ such that
\[u_i \cup u_j = c_{ij}v~ \text{ for }~ 1\leq i\leq j\leq n\]
is called the \emph{cellular cup product matrix}. 
\end{enumerate}

\end{definition}

In general, choosing different orientations of $S^2_i$'s and $X_P$ will affect the signs of homology classes $e_i\in H_2(X_P;\QQ)$ and $f\in H_4(X_P;\QQ)$, resulting in a different cellular cup product matrix $M_{cup}(X_P)$.
In~\cite{FSS2}, specific orientations of $S^2_i$'s and $X_P$ are defined and the authors showed that the entries of $M_{cup}(X_P)$ are given by \eqref{eq_Mcup_entry} below. 
For the generator of $H_4(X_P;\QQ)$, we take $f$ as the fundamental class of $X_P$. Regarding the choices of orientations of $S^2_i$'s, see~Appendix~\ref{Appendix_alg_cel_basis}.

From now on, the cellular basis $\{u_1,\ldots,u_n;v\}$ always means the basis with the specific choice of orientations given in Appendix~\ref{Appendix_alg_cel_basis}.\footnote{Such cellular basis is called the \emph{algebraic cellular basis} in~\cite{FSS2}.}

\begin{theorem}{\cite[Theorem 7.5]{FSS2}}\label{thm_M_cup}
Let $X_P$ be a compact toric surface and write $\lambda_i=(a_i, b_i)\in \ZZ^2$ for $1\leq i\leq n+2$ the primitive vectors generating $1$-dimensional cones of the normal fan $\Sigma$ of $P$. By a suitable change of the basis of $\ZZ^2$, we assume that $\lambda_{n+1}=(1,0)$. 
Then
the cellular cup product matrix $M_{cup}(X_P)=(c_{ij})_{1\leq i,j \leq n}$ with respect to the cellular basis $\{u_1,\ldots,u_n;v\}$ is given by
\begin{equation}\label{eq_Mcup_entry}
c_{ij}=\begin{cases}
\displaystyle\frac{b_j(a_ib_{n+2}-a_{n+2}b_i)}{b_{n+2}}  &\text{if }i\leq j\\
\displaystyle\frac{b_i(a_jb_{n+2}-a_{n+2}b_j)}{b_{n+2}}  &\text{if }i> j.
\end{cases}\end{equation}
\end{theorem}

We note that $M_{cup}(X_P)$ is symmetric. Moreover, 
if $X_P$ has a smooth fixed point, or equivalently $\lambda_{n+1}=(1,0)$ and $\lambda_{n+2}=(0,1)$ after a suitable basis change of~$\mathbb{Z}^2$, then Equation \eqref{eq_Mcup_entry} has a simpler form 
\begin{equation}\label{eq_Mcup_entry_sm_vert}
c_{ij}=a_ib_j,
\end{equation}
implying that $u_i \cup u_j = a_ib_j v$ for $1\leq i\leq j\leq n$. 

\begin{example}\label{ex_M_cup}
For the fan $\Sigma$ considered in Example \ref{ex_surface}, we apply \eqref{eq_Mcup_entry_sm_vert} to compute the cellular cup product matrix $M_{cup}(X_P)$ and obtain 
\[
M_{cup}(X_P)=\begin{pmatrix}
-2&2&4\\
2&2&4\\
4&4&-2
\end{pmatrix}.
\]
\end{example}

\section{Relation between cellular product matrices and intersection product matrices}\label{sec_cellular and poincare}
In Examples \ref{ex_surface} and \ref{ex_M_cup}, 
we have computed the intersection product matrix
and the cellular cup product matrix
for the toric surface associated with the fan in Figure~\ref{fig_2dim_fan_ex}. It is easy to check that these two matrices 
satisfy the claim of Theorem~\ref{thm_main_Intro}, that is,  
\[
M_{cup}(X_P)\cdot M_{int}(X_P)=I.
\]
Below, we demonstrate that the above equation holds for any toric surface. In what follows, we shall keep using $X_P$ for the toric surface associated with a rational polygon $P$. Hence, whenever we need to refer to the discussion of Section \ref{sec_toric_surf}, we consider the toric variety $X_\Sigma$ corresponding to the normal fan $\Sigma$ of $P$.

In the above equation, the matrix $M_{cup}(X_P)$ records the products of cellular basis elements in $H^*(X_P;\QQ)$, while $M_{int}(X_P)$ records the products of elements in the basis~\eqref{eq_PD_def} for $A^*(X_P)_{\QQ}$.
To compare them, we first identify the Chow ring with the cohomology ring using the isomorphism $A^i(X_P)_{\QQ}\to H_{4-2i}(X_P;\QQ)$ and the Poincar\'{e} duality
\begin{equation}\label{eq_PD}
\PD\colon H^*(X_P;\QQ) \to H_{4-*}(X_P;\QQ)
\end{equation}
defined as follows.

Let $\triangle_*\colon H_i(X_P;\QQ)\to\bigoplus_{p+q=i}H_p(X_P;\QQ)\otimes H_q(X_P;\QQ)$ be the comultiplication induced by the diagonal map $\triangle\colon X_P\to X_P\times X_P$, and let
\[
\langle ~ , ~\rangle\colon H^i(X_P,\QQ)\otimes H_j(X_P;\QQ)\to\QQ
\]
be the evaluation map if $i=j$ and be the zero map if $i\neq j$.
Then the Poincar\'{e} duality $\PD$ is defined to be the composite
\begin{align*}
H^i(X_P;\QQ)
&\xrightarrow{id \otimes f}H^i(X_P;\QQ)\otimes H_4(X_P;\QQ)\\
 &\xrightarrow{id\otimes \Delta_*} 
 H^i(X_P;\QQ)\otimes \left( \bigoplus_{p+q=4} H_p(X_P;\QQ)\otimes H_q(X_P;\QQ)\right) \\
 &\xrightarrow{\langle  \,, \,\rangle\otimes id} \QQ\otimes H_{4-i}(X_P;\QQ)\cong H_{4-i}(X_P;\QQ) 
\end{align*}
where $f\in H_4(X_P;\QQ)$ in the first arrow is the fundamental class of $X_P$.
It is known that $\PD$ is an isomorphism when $X_P$ is a toric surface; see~\cite{Sat}.

\begin{definition}\label{def_Poincare dual basis}
Let $X_P$ be the toric surface associated with the normal fan~$\Sigma$ of a rational polygon~$P$, and let
\[
\{[D_{\rho_1}],\ldots,[D_{\rho_n}];[V]\}
\]
be the basis~\eqref{eq_PD_def}
for $A^{>0}(X_P)_{\QQ}$.
Under the isomorphism $A^*(X_P)_{\QQ}\cong H_{4-*}(X_P;\QQ)$, identify $[D_{\rho_1}],\ldots,[D_{\rho_n}]$ and $[V]$ with homology classes of their corresponding subvarieties in $X_P$. Then we call the collection
\[
\left\{\PD^{-1}([D_{\rho_1}]),\ldots,\PD^{-1}([D_{\rho_n}]);\PD^{-1}([V])\right\}
\]
the \emph{Poincar\'{e} dual basis} for $\widetilde{H}^*(X_P;\QQ)$.
\end{definition}

The multiplications of Poincar\'{e} dual basis elements are given by $M_{int}(X_P)$. To be more precise, we have
\begin{multline}
    \label{eqn intersection matrix for Poincare basis} 
\PD^{-1}([D_{\rho_i}])\cup\PD^{-1}([D_{\rho_j}])
=\psi([D_{\rho_i}])\cup \psi([D_{\rho_j}])\\
=\psi([D_{\rho_i}]\cdot [D_{\rho_j}])
=m_{ij}\psi([V])=m_{ij}\PD^{-1}([V])
\end{multline}
where $m_{ij}$ is the $(i,j)$-entry of $M_{int}(X_P)$ given in~\eqref{eq_M_int_mij}, and $\psi$ is the ring isomorphism from~\eqref{eq_chow_to_singular}.

Now there are two bases for $\widetilde{H}^*(X_P;\QQ)$, namely the cellular basis given in Appendix~\ref{Appendix_alg_cel_basis}, and the Poincar\'{e} dual basis given in Definition~\ref{def_Poincare dual basis}. We compare these two bases.

\begin{lemma}\label{lemma_PD_of_x_i}
Let $\{u_1,\ldots,u_n;v\}$ and $\left\{\PD^{-1}([D_{\rho_1}]),\ldots,\PD^{-1}([D_{\rho_n}]);\PD^{-1}([V])\right\}$ be the bases for $\widetilde{H}^*(X_P;\QQ)$ given in Appendix~\ref{Appendix_alg_cel_basis} and Definition~\ref{def_Poincare dual basis} respectively, and let $c_{ij}$ be the $(i,j)$-entry of the cellular cup product matrix $M_{cup}(X_P)$. Then
\begin{equation}\label{eqn_divisor = e_i deg 2}
u_i=\sum^n_{j=1}c_{ij}\PD^{-1}([D_{\rho_j}])\in H^{2}(X_P;\QQ)
\qquad
\text{for }1\leq i\leq n,
\end{equation}
and
\begin{equation}\label{eqn_divisor = e_i deg 4}
v=\PD^{-1}([V])\in H^4(X_P;\QQ).
\end{equation}
\end{lemma}

\begin{proof}
Here we prove~\eqref{eqn_divisor = e_i deg 4}. 
The proof of~\eqref{eqn_divisor = e_i deg 2} will be given in Appendix \ref{appendix_B}.

By definition, the rational equivalence class $[V]$ is $[V(\sigma)]$ for any $2$-dimensional cone $\sigma$ in $\Sigma$, whose corresponding subvariety in $X_P$ is a single point. Therefore, the homology class $[V]$ is a generator of $H_0(X_P;\QQ)$, and we have
\[
\Delta_*(f)=[V]\otimes f+f\otimes[V]+\sum_{i,j=1}^n d_{ij}e_i\otimes e_j
\]
for some $d_{ij}\in \QQ$ and a set of basis $\{e_1,\ldots, e_n\}\subset H_2(X_P;\QQ)$.
Apply $\PD$ to $v$ to~get
\[\begin{split}
\PD(v)&=(\langle ~,~\rangle\otimes id)(v\otimes \Delta_* f)
=\langle v,f\rangle\otimes [V].
\end{split}
\]
Since $v\in H^4(X_P;\QQ)$ is the Kronecker dual of the fundamental class $f\in H_4(X_P;\QQ)$, the evaluation $\langle v,f\rangle=1$ and hence $\PD(v)=[V]$. The fact that $\PD$ is an isomorphism then implies~\eqref{eqn_divisor = e_i deg 4}. 
\end{proof}

Below, we state our main result. For clarification, we will first summarize our notation again. Let $X_P$ be a toric surface associated with a lattice polygon $P$ in~$M\otimes_\ZZ \RR$. For the set $\Sigma^{(1)}=\{\rho_1, \dots, \rho_{n+2}\}$ of $1$-dimensional cones of the normal fan $\Sigma$,  we write the primitive vector generating $\rho_i$ as $\lambda_i=(a_i,b_i)\in N(\cong \ZZ^2)$. By a suitable change of basis of $N$, we assume that 
\begin{equation}\label{eqn_thm condition on lambda}
\lambda_{n+1}=(1,0).    
\end{equation}

\begin{theorem}\label{thm_main}
Let $X_P$ be a toric surface associated with a lattice polygon $P$ satisfying Condition~\eqref{eqn_thm condition on lambda}.
Let $M_{int}(X_P)$ be the intersection product matrix with respect to the basis
\[
\{[D(\rho_1)],\ldots,[D(\rho_n)];[V]\}\subset A^{>0}(X_P)_{\QQ}
\]
in~\eqref{eq_M_int_mij}, and let $M_{cup}(X_P)$ be the cellular cup product matrix with respect to the cellular basis
\[
\{u_1,\ldots,u_n;v\}\subset \widetilde{H}^*(X_P;\QQ)
\]
defined in Appendix~\ref{Appendix_alg_cel_basis}. Then we have
\[
M_{int}(X_P)\cdot M_{cup}(X_P)=I.
\]
\end{theorem}
\begin{proof}
For $1\leq i, j\leq n$, let $c_{ij}$ and $m_{ij}$ be the $(i,j)$-entries of $M_{cup}(X_P)$ and $M_{int}(X_P)$, respectively. Consider the following string of equalities
\begin{eqnarray*}
c_{ij}v &=&u_i\cup u_j\\
&=&\left(\sum^n_{k=1}c_{ik}\PD^{-1}
([D_{\rho_k}])\right)\,\cup\,\left(\sum^n_{l=1}c_{jl}\PD^{-1}([D_{\rho_l}])\right)\\
&=&\sum^n_{k,l=1}c_{ik}c_{jl}\cdot
\PD^{-1}([D_{\rho_k}])\cup\PD^{-1}([D_{\rho_l}])\\
&=&\sum^n_{k,l=1}c_{ik}c_{jl} m_{kl}\PD^{-1}([V])\\
&=&\sum^n_{k,l=1}c_{ik}c_{jl} m_{kl}v
\end{eqnarray*}
where the second and the last lines are due to Lemma~\ref{lemma_PD_of_x_i} and the fourth line is due to \eqref{eqn intersection matrix for Poincare basis}.
Comparing the coefficients on both sides, we rewrite the last equality as
\[
M_{cup}(X_P)=M_{cup}(X_P)\cdot M_{int}(X_P)\cdot M_{cup}(X_P)
\]
Since $M_{cup}(X_P)$ is invertible, we obtain
\[
M_{int}(X_P)\cdot M_{cup}(X_P)=I.\qedhere
\]
\end{proof}

\appendix
\section{Construction of the cellular bases}\label{Appendix_alg_cel_basis}
In this appendix, we outline the construction of the cellular basis over $\QQ$, which gives the desired entries of $M_{cup}(X_P)$ in Theorem~\ref{thm_M_cup}. The readers are referred to~\cite[Section 6]{FSS2} for more details.

First, recall the definition of the degree-4 generator $v\in H^4(X_P;\QQ)$. Orient $P$ and $T^2$ as subspaces of $\RR^2$ and $\CC^2$ respectively, and orient $X_P$ via the orbit map $$P\times T^2\to X_P.$$
Let $[X_P]\in H_4(X_P;\ZZ)$ be the fundamental class of $X_P$ with respect to this orientation. Define $v\in H^4(X_P;\QQ)$ to be the Kronecker dual of $f=[X_P]\otimes 1\in H_4(X_P;\QQ)$.

For the cellular basis $\{u_1, \dots, u_n\}$ of $H^2(X_P;\QQ)$, we first consider the case when
\begin{equation}\label{eq_smooth_vert}
\lambda_{n+1}=(1,0) \quad \text{and}\quad \lambda_{n+2}=(0,1).
\end{equation}
In this case, the homotopy cofibration \eqref{eq_cofib} becomes 
\[
S^3 \to \bigvee_{i=1}^n S^2_i \to X_P.
\]
Recall that $S^2_i$ is the 2-sphere fixed by the circle subgroup of $T^2$ generated by $\lambda_i=(a_i, b_i)$. The homology class $e_i\in H_2(X_P;\QQ)$ corresponds to $[S^2_i]\in H_2(S^2_i;\QQ)$, and $u_i\in H^2(X_P;\QQ)$ is the Kronecker dual to $e_i$. The only ambiguity arises from the choice of orientation on $S^2_i$, which determines the sign of $e_i$.

For each $i=1,\ldots,n$, we define the basis element $u_i$ based on two cases:
(1) when $\lambda_i=(a_i, b_i)$ satisfies $a_ib_i\neq 0$, and (2) when $a_ib_i=0$.
The strategy involves analysing the following diagram of homotopy cofibrations
\begin{equation}\label{eq_diag_cofib_pinch}
\begin{tikzcd}
S^3 \arrow{r}\arrow{d}{=} & \bigvee_{i=1}^n S^2_i \arrow{d}{\text{pinch}} \arrow{r} &X_P\arrow{d}{\eta_i} \\
S^3 \arrow{r}{f} & S^2_i  \arrow{r}& Y
\end{tikzcd}
\end{equation}
where $Y$ is the homotopy cofiber of $f$ and $\eta_i$ is an induced map.
 Note that $Y$ may not be a toric variety, 
in fact:
 \begin{enumerate}
     \item If $a_ib_i\neq 0$, $Y$ is a toric orbifold (see~\cite[Section 7]{DJ});
     \item If $a_ib_i= 0$, $Y$ is a degenerate toric space (see~\cite[Section 5]{FSS2}).
 \end{enumerate}
The cellular basis element $u_i\in H^2(X_P;\QQ)$ is then defined to be
\begin{equation}\label{eq_u_i_aibi_nonzero}
u_i=\eta^*_i(u^{\triangle}_{(a_i,b_i)}),
\end{equation}
where $u^{\triangle}_{(a_i,b_i)}$ is an integral generator of $H^2(Y;\ZZ)$ defined below.

Now we define $u^{\triangle}_{(a_i,b_i)}$ for the first case, where $\lambda_i=(a_i, b_i)$ satisfies $a_ib_i\neq 0$. Let~$\triangle$ be the triangle in $\RR^2$ with facets labeled by  $(a_i, b_i), (1,0)$ and $(0,1)$, as described in Figure \ref{fig_char_pair_X_tri_sq}(A). 
This labeling is a \emph{rational characteristic function} in the sense of \cite[Definition 3.1]{BSS}; see also \cite[Section 7]{DJ}. Using this, one can define a \emph{toric orbifold}
\[
X_{(a_i, b_i)} \colonequals  \triangle \times T^2 / _\sim
\]
where the equivalence relation $\sim$ is as in~\eqref{eq_top_model}.  In this case, we have $Y=X_{(a_i, b_i)}$ and the bottom row of~\eqref{eq_diag_cofib_pinch} becomes the homotopy cofibration
\[
S^3 \to S^2_{i} \to X_{(a_i, b_i)}.
\]
We define $u^\triangle_{(a_i, b_i)}\in H^2(X_{(a_i,b_i)};\ZZ)$ as the cohomology class corresponding to $S^2_i$. 
Since there is no canonical choice on the orientation of $S^2_i$, the element 
$u^\triangle_{(a_i, b_i)}$ is defined up to sign. Below, we fix this choice using an algebraic method.

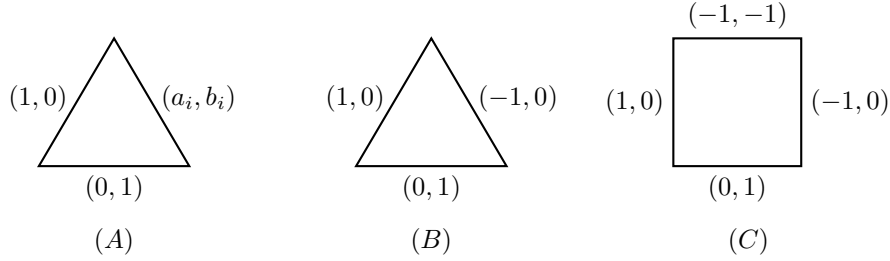
\begin{figure}
\begin{tikzpicture}
\draw[thick] (0,0)--(2,0)--(1,1.7)--cycle;
\node[left] at (1/2, 0.9) {$(1,0)$};
\node[right] at (3/2, 0.9) {$(a_i, b_i)$};
\node[below] at (1,0) {$(0,1)$};
\node  at (1,-1) {$(A)$};

\begin{scope}[xshift=120]
\draw[thick] (0,0)--(2,0)--(1,1.7)--cycle;
\node[left] at (1/2, 0.9) {$(1,0)$};
\node[right] at (3/2, 0.9) {$(-1,0)$};
\node[below] at (1,0) {$(0,1)$};
\node  at (1,-1) {$(B)$};
\end{scope}

\begin{scope}[xshift=240]
\draw[thick] (0,0)--(1.7,0)--(1.7,1.7)--(0,1.7)--cycle;
\node[below] at (1.7/2, 0) {$(0,1)$};
\node[above] at (1.7/2, 1.7) {$(-1,-1)$};
\node[left] at (0,1.7/2) {$(1,0)$};
\node[right] at (1.7,1.7/2) {$(-1,0)$};
\node  at (1,-1) {$(C)$};
\end{scope}

\end{tikzpicture}
\caption{(A) A rational characteristic pair; (B) a degenerate characteristic pair; (C) a characteristic pair.}
\label{fig_char_pair_X_tri_sq}
\end{figure}

Following the result of \cite[Corollary 4.12]{DKS}, together with similar calculations provided in \cite[Example 5.4]{BSS}, there is an injective homomorphism
\begin{equation*}\label{eq_Psi_cohom_Xab}
\Psi \colon H^\ast(X_{(a_i, b_i)};\ZZ) \to \QQ[y_1, y_2, y_3]/\left<y_1y_2y_3, a_iy_1+y_2, b_iy_2+ y_3\right>
\end{equation*}
where the image of $\Psi$ is generated by $[a_ib_iy_1]$ in degree $2$ and $[y_2y_3]$ in degree $4$.  
We define the element $u^\triangle_{(a_i, b_i)}\in H^2(X_{(a_i, b_i)};\ZZ)$ by
\begin{equation}\label{eq_u_i_orientation}
\Psi(u^\triangle_{(a_i, b_i)})=[a_ib_iy_1].
\end{equation}

Next, we consider the second case, where $\lambda_i=(a_i,b_i)$ satisfies $a_ib_i=0$. By the primitivity of $\lambda_i$, it suffices to assume $\lambda_i=(a_i, b_i)$ is $(-1, 0)$ or $(0,-1)$. 
For $\lambda_i=(-1,0)$, let $\triangle$ be the triangle in $\RR^2$ with facets assigned $(1,0), (0,1)$ and $(-1,0)$, as described in Figure \ref{fig_char_pair_X_tri_sq}\,(B). The
two edges intersecting the top vertex are assigned parallel vectors, which yields a \emph{degenerate characteristic function} in the sense of~\cite[Definition 5.1]{FSS2}. From this setup, one can construct the space
\begin{equation*}\label{eq_cofib_deg_toric}
Y_{(-1,0)}\colonequals \triangle\times T^2 /_{\sim_d}
\end{equation*}
where $\sim_d$ is defined similarly  to~\eqref{eq_top_model} together with an additional condition collapsing $\{v\}\times T^2$ to a point for each vertex $v$ of $\triangle$. 
It is shown in~\cite[Remark 5.9]{FSS2} that~$Y_{(-1,0)}$ is homotopy equivalent to $S^4\vee S^2$. 

Furthermore, in the bottom row of~\eqref{eq_diag_cofib_pinch}, we have $Y=Y_{(-1,0)}$ and the homotopy cofibration
\[
S^3\to S^2 \to Y_{(-1,0)}\simeq S^4\vee S^2.
\]
Let $u^\triangle_{(-1,0)}\in H^2(Y_{(-1,0)};\ZZ)$ be the cohomology class corresponding to $S^2$. 
Note that it cannot be defined directly via an analogue of~\eqref{eq_u_i_orientation}, because
$Y_{(-1,0)}$ is not a toric orbifold and the homomorphism $\Psi$ does not exist.
Instead, we define $u^\triangle_{(-1,0)}$ using a pullback construction from another toric orbifold.

Let $\Box$ be the square with vectors assigned on its facets as described in Figure~\ref{fig_char_pair_X_tri_sq}\,(C). The
corresponding space 
\[
X_{(-1,0),(-1,-1)}\colonequals \Box \times T^2/_\sim
\]
is homeomorphic to $\CP^2\# \overline{\CP^2}$; see~\cite[Example 1.19]{DJ}. 
Consider the diagram
\[
\begin{tikzcd}
S^3 \arrow{r}\arrow{d}[swap]{=}&  S^2_{(-1,0)}\vee S^2_{(-1,-1)}  \arrow{r} \arrow{d}[swap]{p}    &X_{(-1,0), (-1,-1)}\arrow{d}[swap]{\eta}\\
S^3 \arrow{r} &S^2_{(-1,0)} \arrow{r} & Y_{(-1,0)}
\end{tikzcd} 
\]
where the rows are homotopy cofibration sequences. The spheres $S^2_{(-1,0)}$ and $S^2_{(-1,-1)}$ are the $2$-spheres embedded in $X_{(-1,0),(-1,-1)}$ that are fixed by the circle subgroups of $T^2$ generated by $(-1,0)$ and $(-1,-1)$, respectively.
From \cite[Theorem 4.14]{DJ} or \cite{Dan, Jur}, we have a ring isomorphism 
\[
\Theta \colon H^\ast(X_{(-1,0),(-1,-1)};\ZZ) \to \frac{\ZZ[z_1, z_2, z_3, z_4]}{ \left<z_1z_3,z_2z_4, -z_1-z_2+z_3, -z_2+z_4\right>}. 
\]
Define $u^{\triangle}_{(-1,0)}\in H^2(Y_{(-1,0)};\ZZ)$ by the equation
\begin{equation}\label{eq_Theta_u_square}
\left(\Theta\circ\eta_1^*\right)(u^\triangle_{(-1,0)})= [z_2] (=[z_4] = [-z_1+z_3]).
\end{equation}

When $\lambda_i=(0,-1)$,  define $u_{(a_i,b_i)}^\triangle$ by repeating the above procedure with $\triangle$ assigned $(1,0), (0,1)$ and $(0,-1)$, and $\Box$ assigned $(1,0), (0,1)$ and $(-1,-1), (0,-1)$. This completes defining $u_i\in H^2(X_P;\QQ)$ using \eqref{eq_u_i_aibi_nonzero} under the assumption~\eqref{eq_smooth_vert}.

\begin{remark}\label{rmk_appendix}
One can also consider the toric manifold corresponding to the rectangle~$\Box$ with the assignment $(1,0), (0,1)$ and $(-1,0), (0,-1)$ on each facet of~$\Box$ to specify orientations of the invariant $2$-spheres $S^2_{(-1,0)}$ and $S^2_{(0,-1)}$. This convention is compatible with the definitions described above. The readers are referred to \cite[Subsection 6.1]{FSS2}.
\end{remark}

Finally, we define $u_i\in H^2(X_P;\QQ)$ in general case. One can still assume that $\lambda_{n+1}=(1,0)$ by an appropriate choice of the basis of the lattice $N$. Then, we consider a linear endomorphism $\kappa$ on $N$ defined by 
\[
\kappa(x,y)=(b_{n+2}x-a_{n+2}y, a_{n+2}y),
\] 
which determines a new fan $\Sigma'$ with rays spanned by $\kappa(\lambda_i)$ for $i=1, \dots, n+2$. 
Note that the primitive vectors generating $\kappa(\lambda_{n+1})$ and $\kappa(\lambda_{n+2})$ are $(1,0)$ and $(0,1)$, respectively. Therefore, the toric variety $X_{\Sigma'}$ satisfies the condition \eqref{eq_smooth_vert} and one can define the cellular basis $\{u_1',\dots, u_n'\}$ of $H^2(X_{\Sigma'};\QQ)$. 

The linear map $\kappa$ induces a morphism $\Sigma\to \Sigma'$ of fans, which induces a morphism of toric varieties 
\begin{equation}\label{eq morphism kappa}
\tilde \kappa \colon X_\Sigma \to X_{\Sigma'};
\end{equation}
see~\cite[Section 3.3]{CLS}.
Define
\[
u_i=\frac{g_i}{a_{n+2}b_{n+2}}\tilde{\kappa}^*(u'_i)
\]
where $g_i=\gcd(a_ib_{n+2},a_{n+2}b_i)$.
This definition implies that the orientations of the~$S^2_i$'s in $X_P$ are given by those of $S^2_i$'s in $X_{\Sigma'}$.
The readers are referred to \cite[Section 7]{FSS2} for more details.

\section{Proof of Lemma \ref{lemma_PD_of_x_i}}\label{appendix_B}
Here we complete the proof of Lemma~\ref{lemma_PD_of_x_i} by showing Equation~\eqref{eqn_divisor = e_i deg 2}. To begin with, we study the images of the cellular basis elements $u_1,\ldots,u_n\in H^2(X_P;\QQ)$ under the Poincar\'{e} duality $\PD\colon H^2(X_P;\QQ)\to H_2(X_P;\QQ)$.
Let $e_1,\ldots,e_n\in H_2(X_P;\QQ)$ be the Kronecker duals of $u_1,\ldots,u_n$. That is, they are defined by the equations
\[
u_i(e_j)=\delta_{ij}.
\]
We call the set $\{e_1,\ldots,e_n\}$ the \emph{Kronecker dual basis} for $H_2(X_P;\QQ)$.

\begin{lemma}\label{prop matrix rep PD}
Let $X_P$ be the toric surface associated with a lattice polygon $P$ that satisfies Condition~\eqref{eqn_thm condition on lambda}.
Let $u_1, \dots, u_n\in H^2(X_P, \mathbb{Q})$ be the degree-2 cellular basis elements, and let $e_1, \dots, e_n\in H_2(X_P;\mathbb{Q})$ be the Kronecker dual basis.
Then
\begin{equation*}\label{eqn_PD_matrix_repr}
\PD(u_i)=\sum_{j=1}^n c_{ij}e_j
\end{equation*}
where $c_{ij}$ is the $(i,j)$-entry of the cellular cup product matrix $M_{cup}(X_P)$ in \eqref{eq_Mcup_entry}. 
\end{lemma}
\begin{proof}
Consider the string of equations
\[
\left\langle{u_k\otimes u_l,\Delta_*(f)}\right\rangle
=\left\langle{\Delta^*(u_k\otimes u_l),f}\right\rangle
=\left\langle{u_k\cup u_l,f}\right\rangle
=\left\langle{c_{kl}v,f}\right\rangle
=c_{kl}.
\]
The second equality is due to the definition of cup products, the third equality is due to Definition~\ref{def_M_cup}~(2), and the last equality is due to the fact that $v\in H^4(X_P;\QQ)$ is the Kronecker dual of the fundamental class $f\in H_4(X_P;\QQ)$.
This implies
\[
\Delta_*(f)=1\otimes f+f\otimes 1+\sum^n_{i,j=1}c_{ij}e_i\otimes e_j.
\]
Applying the Poincar\'e duality $\PD$ in \eqref{eq_PD} to $u_k$, we obtain 
\begin{align*}
\PD(u_k)&=\left( (\langle ~,~ \rangle\otimes id)\circ (id\otimes \Delta_\ast)\right) (u_k\otimes f) \label{eq_base_change_calculation} \\
&= (\langle ~,~ \rangle\otimes id)
\left( u_k\otimes\left(1\otimes f+f\otimes 1+\sum^n_{i,j=1} c_{ij} e_i\otimes e_j\right) \right)  \\
&=\sum^n_{i,j=1}\langle{u_k,c_{ij}e_i}\rangle\otimes e_j =\sum^n_{i,j=1}c_{ij}\delta_{ki} e_j=\sum^n_{j=1}c_{kj}e_j.
\end{align*}
Hence, the result follows. 
\end{proof}

Due to Lemma~\ref{prop matrix rep PD}, proving Equation~\eqref{eqn_divisor = e_i deg 2} is equivalent to showing
\begin{equation}\label{eqn_divisor = e_i}
[D_{\rho_i}]=e_i
\end{equation}
for $1\leq i\leq n$, where $[D_{\rho_i}]$ is regarded as the homology class of the subvariety corresponding to $D_{\rho_i}$.
One can see that this subvariety is homeomorphic to the 2-sphere $\pi^{-1}(F_i)\simeq S^2_i$ in~\eqref{eq_cofib}, which is represented by $e_i$. Therefore, \eqref{eqn_divisor = e_i} holds up to sign. To determine the sign, one compares the orientations of $S^2_i$ corresponding to $D_{\rho_i}$ and $e_i$. The subvariety $D_{\rho_i}$ is equipped with the orientation induced from the orientation of $X_P$ as a complex submanifold, 
while the orientation of $e_i$ is interpreted as the choice of its dual generator $u_i\in H^2(X_P;\QQ)$ which is defined in Appendix~\ref{Appendix_alg_cel_basis}.
The following computation uses algebraic methods to show that these two orientations are the same.

Consider the composite
\[
H_2(X_P;\mathbb{Q}) \xrightarrow{\PD^{-1}} 
H^2(X_P;\mathbb{Q}) \xrightarrow{\Phi}
\mathbb{Q}[x_1, \dots, x_{n+2}]/ \left(\mathcal{I}+\mathcal{J}\right),
\]
where $\Phi$ is the restriction of isomorphism~$\varphi$ to $H^2(X_P;\QQ)$ (see~\eqref{Eq_chow_SR_surface}). Note that 
\[
(\Phi\circ \PD^{-1})([D_{\rho_i}]) = [x_i];
\]
see \eqref{eq_chow_to_singular},  \eqref{eq_rational_cohom_SR_presentation} and \eqref{Eq_chow_SR_surface}. Hence, we prove the assertion by showing that 
\begin{equation*}\label{eq_Lemma_subclaim}
(\Phi\circ \PD^{-1})(e_i)=[x_i]
\end{equation*}
for $i=1, \dots, n$. Since we have 
\[
\PD^{-1}(e_i)= \sum_{j=1}^n d_{ij}u_j 
\]
for the entries $d_{ij}$ of $M_{cup}(X_P)^{-1}$ by Lemma \ref{prop matrix rep PD}, it suffices to verify that 
\begin{equation}\label{eq_Phi_u_i}
\Phi(u_i) =\left[  \sum_{j=1}^n c_{ij}x_j\right] 
\end{equation}
for the entries $c_{ij}$ of $M_{cup}(X_P)$ given in \eqref{eq_Mcup_entry}. We will first prove \eqref{eq_Phi_u_i} for the case where $X_P$ has at least one smooth fixed point. Then, we extend this result to general cases.

When $X_P$ has a smooth fixed points, we may assume that $\lambda_{n+1}=(1,0)$ and $\lambda_{n+2}=(0,1)$ by an appropriate change of the basis of $N$. First we consider $\rho_i$ generated by $\lambda_i=(a_i, b_i)$ with $a_ib_i\neq 0$. In this case, the map
 \[\text{$\eta_i\colon X_P \to X_{(a_i, b_i)}$}\]
 given in~\eqref{eq_diag_cofib_pinch} induces the diagram
\begin{equation*}\label{eq_diag_edge_cont}
\begin{tikzcd}
H^2(X_P;\mathbb{Q}) \rar{\Phi}& \frac{\mathbb{Q}[x_1, \dots, x_{n+2}]}{\mathcal{I}+\mathcal{J}}\\
H^2(X_{(a_i, b_i)};\mathbb{Q})\rar{\Psi}\uar{\eta_i^\ast} & \frac{\mathbb{Q}[y_1, y_2, y_3]}{\left<y_1y_2y_3, a_iy_1+y_2, b_iy_2+y_3\right>}.\uar{\widetilde{\eta}_i^\ast}\end{tikzcd}
\end{equation*} 
We refer to \cite[Section 3, Section 4.2]{FSS2}
for the commutativity of the above square. 
Hence, the cohomology class $u_i\in H^2(X_P, \mathbb{Q})$ satisfies that  
\[
\Phi(u_i)=\left(\Phi\circ \eta_i^\ast\right)(u_i^\triangle) =\left(\widetilde{\eta}_i^\ast \circ\Psi\right)(u_i^\triangle) = \left[  \sum_{j=1}^n c_{ij}x_j\right],
\]
where the first equality is due to the construction \eqref{eq_u_i_aibi_nonzero} and the last equality follows from \cite[Lemma 4.8]{FSS2}. 

Next, we consider $\rho_i$ generated by $\lambda_i=(-1,0)$. Since $\Sigma$ is a fan, there exists $j$ with $i< j \leq n$ such that $\lambda_j=(0, -1)$ or $\lambda_j=(a_j, b_j)$ with $a_jb_j\neq 0$ and $b_j<0$. For simplicity, we assume~$i=1$ and~$j=2$. When $\lambda_j=(0,-1)$, 
the cofibration $S^3 \to \bigvee_{i=1}^n S^2_i \to X_P$ and the pinch map $\bigvee_{i=1}^n S^2_i \to S^2_1 \vee S^2_2$ induce a map 
\[
\eta_{12}\colon X_P \to X_{(-1,0), (0,-1)}.
\]
Therefore, we have the following commutative diagram 
\[
\begin{tikzcd}
    H^\ast(X_{(-1,0), (0,-1)})
    \dar{\cong}[swap]{\Theta}\rar{\eta_{12}^\ast} & H^\ast(X_P) \dar{\cong}[swap]{\Phi} \\
    \frac{\mathbb{Q}[x_1, x_2, x_3, x_4]}{\left<z_1z_3,z_2z_4, -z_1+z_3, -z_2+z_4\right>} \rar{\widetilde{\eta}_{12}^\ast} &\frac{\mathbb{Q}[x_1, \dots, x_4]}{\mathcal{I} + \mathcal{J}}.
\end{tikzcd}
\] 
Denote by $u^\square_1\in H^\ast(X_{(-1,0), (0,-1)})$ the first cellular basis element~$u_1$ when $X_P$ is $X_{(-1,0), (0,-1)}$, as described in~Appendix~\ref{Appendix_alg_cel_basis}.
Recall from \eqref{eq_Theta_u_square} and Remark \ref{rmk_appendix} that the first vertical map $\Theta$ maps $u^\square_1$ to~$[x_2]$.
Then we have
\begin{align*}
\Phi(u_1)&=\left(\Phi\circ \eta_{12}^*\right)(u_1^\square)=\left(\widetilde{\eta}_{12}^*\circ \Theta\right)(u_1^\square)=\widetilde{\eta}_{12}^*([x_2])\\
&=[x_2-b_3x_3-\cdots - b_nx_n]
\end{align*}
where the last equality follows from computations similar to the proof of \cite[Lemma 4.8]{FSS2}. Thus we get \eqref{eq_Phi_u_i} for this case as well. 

Suppose $\lambda_j=(a_j, b_j)$ satisfies $a_jb_j\neq 0$ and $b_j<0$, regarding $i=1$ and $j=2$ again for simplicity. In that case, we consider the lattice $N'$ generated by $\left\{\left(\frac{1}{|b_2|}, 0\right), \left( 0, \frac{1}{|a_2|}\right)\right\}$, that is obtained by subdividing the first coordinate and the second coordinate of $N$ into~$|b_2|$ and~$|a_2|$ pieces, respectively. We note that $N$ is a sublattice of $N'$.

Given a lattice polytope $P\in M\otimes_\mathbb{Z} \mathbb{R}$ and its normal fan $\Sigma$ in $N\otimes_\mathbb{Z} \mathbb{R}$, one can consider $\Sigma$ as a fan in $N'\otimes_\mathbb{Z} \mathbb{R}$, which defines a toric variety, say $X_P'$. 
As $\Sigma$ is a fan and $(a_1,b_1)=(-1,0)$, we have $b_2<0$. 
Then the primitive vectors $\lambda_1', \lambda_2' \in N'$ generating $\rho_1$ and $\rho_2$ of $\Sigma$ as a fan in $N'\otimes_\mathbb{Z} \mathbb{R}$ are $(-1,0)$ and $(-1,-1)$. We denote by $(a_i', b_i')\in N'$ the primitive vector generating $\rho_i$ of $\Sigma$ as a fan in $N'\otimes_\mathbb{Z} \mathbb{R}$. To be more precise, since $(a_i, b_i)\in N$ corresponds to 
$(a_i\cdot |b_2|, b_i\cdot |a_2|)\in N'$, we have the following relation
\[
(a_i', b_i')=\left(\frac{a_i\cdot |b_2|}{g_i},\frac{ b_i\cdot |a_2|}{g_i}\right)
\]
for $g_i=\gcd\left\{|a_ib_2|, |b_ia_2|\right\}$ and $i=1,\dots,n$. Note that $g_1=|a_1b_2|$ as $b_1=0$.

Following \cite[Proposition 3.3.7]{CLS}, we have a morphism 
\[
\sigma \colon X_P \to X_P'.
\]
One can also understand the map $\sigma$ as a \emph{rescaling morphism} studied in \cite[Subsection 4.1]{FSS2}. 
Furthermore, the cofibration $S^3 \to \bigvee_{i=1}^n S^2_i \to X_P'$ and the pinch map $\bigvee_{i=1}^n S^2_i \to S^2_1 \vee S^2_2$ induce a map 
\[
\eta_{12}\colon X_P' \to X_{(-1,0), (-1,-1)}.
\]
Therefore, we have the following commutative diagram 
\begin{equation}\label{eq_comm_diag_rescale}
\begin{tikzcd}
    H^\ast(X_{(-1,0), (-1,-1)})
    \dar{\cong}[swap]{\Phi''}\rar{\eta_{12}^\ast} & H^\ast(X_P') \dar{\cong}[swap]{\Phi'}\rar{\sigma^\ast} & H^\ast(X_P) \dar{\cong}[swap]{\Phi}\\
    \frac{\mathbb{Q}[x_1, x_2, x_3, x_4]}{\left<z_1z_3,z_2z_4, -z_1-z_2+z_3, -z_2+z_4\right>} \rar{\widetilde{\eta}_{12}^\ast} &\frac{\mathbb{Q}[x_1, \dots, x_{n+2}]}{\mathcal{I}' + \mathcal{J}'} \rar{\widetilde{\sigma}^\ast} & \frac{\mathbb{Q}[x_1, \dots , x_{n+2}]}{\mathcal{I} + \mathcal{J}},
\end{tikzcd}
\end{equation}
where $\mathcal{I}',\mathcal{I}$ and $\mathcal{J}', \mathcal{J}$ are corresponding ideals as defined in \eqref{eq_rational_cohom_SR_presentation}. 
Again by \eqref{eq_Theta_u_square}, the left vertical map $\Phi''$ maps $u^\square_1\in H^\ast(X_{(-1,0), (-1,-1)})$ to $[x_2]$.

For the bottom composition, we have 
\begin{align*}
\left( \widetilde{\sigma}^\ast\circ \widetilde{\eta}_{12}^\ast\right)  ([x_2]) &=\widetilde{\sigma}^\ast([x_2-b_3'x_3-\cdots - b_n'x_n])\\
&= \left[g_2x_2 - b_3'g_3x_3-\cdots - b_n'g_nx_n\right]  \\
&= -|a_2|\left[b_2x_2+b_3x_3+\cdots + b_nx_n\right], 
\end{align*}
where the first and the second equalities follow from the computation given in the proof of \cite[Lemma 4.8]{FSS2} and from \cite[Lemma 4.4]{FSS2}, respectively. 
For the top composition, we have 
\[
(\sigma^\ast\circ \eta_{12}^\ast) (u_1^\square)=|a_2|u_1
\]
which follows from the result of \cite[Lemma 7.2]{FSS2}. 

Combining the computations above together with the commutativity of \eqref{eq_comm_diag_rescale}, we get 
\begin{align*}
\left(\Phi\circ\sigma^*\circ\eta_{12}^*\right)(u_1^\square)&=\left(\widetilde{\sigma}^*\circ \widetilde{\eta}_{12}^*\circ 
\Phi''\right) (u^\square_1) ; \\
|a_2|\Phi(u_1)&=-|a_2|\left[b_2x_2+b_3x_3+\cdots + b_nx_n\right]; \\
\Phi(u_1)&=-\left[b_2x_2+b_3x_3+\cdots + b_nx_n\right]\\
&=\left[a_1b_2x_2+a_1b_3x_3+\cdots + a_1b_nx_n\right],
\end{align*}
which also verifies \eqref{eq_Phi_u_i}. 

The computation for $\rho_i$ generated by $\lambda_i=(0,-1)$ is similar. 

Finally, we consider a toric variety $X_P$ which does not necessarily have a smooth fixed point. Since each $\lambda_i=(a_i, b_i)$ is primitive, we may assume that $\lambda_i=(1,0)$ and $\lambda_{n+2}=(a_{n+2},b_{n+2})$ with $a_{n+2}b_{n+2}\neq 0$ by a suitable change of the basis of $N$.

We recall $X_{\Sigma'}$ and the morphism $\tilde{\kappa}$ in~\eqref{eq morphism kappa}.
For $X_{\Sigma'}$, let $\lambda'_i=(a'_i,b'_i)$ be the primitive vector that spans the $i$-th ray of $\Sigma'$. Consider the ring isomorphism 
\[
\Phi'\colon H^*(X_{\Sigma'};\QQ)\to
\mathbb{Q}[x_1,\ldots,x_{n+2}]/\left(\mathcal{I}'+\mathcal{J}'\right).
\] 
Applying the preceding calculation for $X_P$ having a smooth fixed point, there are basis elements 
$\{u'_1,\ldots, u'_n\} \subset H^2(X_{\Sigma'};\QQ)$
such that 
\[
\Phi'(u'_i)=\left[\sum_{j=1}^ia'_jb'_ix_j+\sum_{j=i+1}^na'_ib'_{j}x_j\right].
\]

Since both $\Phi$ and $\Phi'$ are isomorphisms, we have a map  
\[
\Phi \circ \widetilde{\kappa}^\ast \circ \Phi'^{-1} \colon 
\mathbb{Q}[x_1, \dots, x_{n+2}]/\left(\mathcal{I}'+\mathcal{J}'\right)
\to 
\mathbb{Q}[x_1, \dots, x_{n+2}]/\left(\mathcal{I}+\mathcal{J}\right)
\]
which we simply denote by $\Upsilon$, namely we have $\Upsilon\circ \Phi' = \Phi \circ \widetilde{\kappa}^*$. We refer to \cite[Section 3]{FSS2}. 
We note that $u_i=\frac{g_i}{a_{n+2}b_{n+2}}\widetilde{\kappa}^*(u'_i)$ and 
$\Upsilon([x_i])=[g_ix_i]$
for $1\leq i\leq n$ by~\cite[Theorem 7.5]{FSS2}. 
Therefore, we have
\[
\begin{split}
    \Phi(u_i)&=\Phi\left(\frac{g_i}{a_{n+2}b_{n+2}}\widetilde{\kappa}^*(u'_i)\right)\\
    &=\frac{g_i}{a_{n+2}b_{n+2}}\left(\Upsilon\circ \Phi'\right) (u'_i)\\
    &=\frac{g_i}{a_{n+2}b_{n+2}}\Upsilon \left(
\left[ \sum_{j=1}^ia'_jb'_ix_j+\sum_{j=i+1}^na'_ib'_{j}x_j\right]\right) \\
    &=\frac{g_i}{a_{n+2}b_{n+2}}
\left[\sum_{j=1}^ia'_jb'_ig_jx_j+\sum_{j=i+1}^na'_ib'_{j}g_jx_j\right]\\
&=\left[\sum_{j=1}^n c_{ij}x_j\right]
\end{split}
\]
for $c_{ij}$'s defined in \eqref{eq_Mcup_entry}, which establishes \eqref{eq_Phi_u_i}.
This implies Equation~\eqref{eqn_divisor = e_i} and hence Equation~\eqref{eqn_divisor = e_i deg 2}. Now the proof of Lemma~\ref{lemma_PD_of_x_i} is complete.

%

\end{document}